\documentclass[leqno,12pt]{article} 
\usepackage{amscd}
\usepackage{bm}
\usepackage{amsmath, amssymb}
\usepackage{amsthm} 
\theoremstyle{plain} 
\newtheorem{theorem}{\indent\sc Theorem}[section]
\newtheorem{lemma}[theorem]{\indent\sc Lemma}
\newtheorem{corollary}[theorem]{\indent\sc Corollary}
\newtheorem{proposition}[theorem]{\indent\sc Proposition}

\theoremstyle{definition} 
\newtheorem{definition}[theorem]{\indent\sc Definition}
\newtheorem{remark}[theorem]{\indent\sc Remark}
\newtheorem{example}[theorem]{\indent\sc Example}

\makeatletter
\def\address#1#2{\begingroup
\noindent\parbox[t]{7.8cm}{%
\small{\scshape\ignorespaces#1}\par\vskip1ex
\noindent\small{\itshape E-mail address}%
\/: #2\par\vskip4ex}\hfill%
\endgroup}%
\makeatother
\title{\uppercase{On the Clifford theorem for surfaces}} 
\author{
\textsc{Hao Sun} 
}
\date{} 
\DeclareMathOperator{\Pic}{Pic}

\begin{document}

\maketitle



\begin{abstract}
We give two generalizations of the Clifford theorem to algebraic
surfaces. As an application, we obtain some bounds for the number of
moduli of surfaces of general type.
\end{abstract}

\section*{Introduction} 
The classical Brill-Noether theory is to study special divisors or
linear systems on an algebraic curve, and the Clifford theorem is
the first step of the theory (cf. \cite{A}). The main purpose of
this paper is to generalize the Clifford theorem to algebraic
surfaces.

Let $X$ be a smooth projective complex surface and $L$ a divisor on
it. One of the fundamental problems in the surface case is to study
the adjoint linear system $|K_X+L|$. Roughly speaking, the behavior
of this linear system depends on the positivity of $L$. When $L$ is
positive, we have a celebrated method of Reider \cite{Re} (see also
\cite{Bom} and \cite{Tan}). When $L$ is zero, the canonical system
has also been studied systematically by Beauville \cite{Be}. When
$L$ is negative, the linear system corresponds to the special
divisors on a curve. Exactly, we say a divisor $D$ on $X$ a special
divisor if it is effective and $h^0(K_X-D)>0$. However, for
surfaces, we have no general method to study such special divisors.
In order to find a powerful method to study special linear systems
in the surface case, we need to establish first a Clifford-type
theorem.

One easy generalization of the Clifford theorem is as follows. Let
$L$ be a special divisor on $X$. From $h^0(L)+h^0(K_X-L)\leq
h^0(K_X)+1$ and the Riemann-Roch theorem, we get
\begin{equation}\label{I}
h^1(L)\leq q+\frac{1}{2}L(K_X-L),
 \end{equation} where $q$ is the irregularity of $X$. If $L=0$ or $K_X$, the equality
holds. As in the curve case, the nontrivial problem is to
characterize the equality. Our first result describes such
conditions on the surface and on the divisor $L$. We can assume that
$L\nsim0$ and $K_X-L\nsim 0$.

\begin{theorem}\label{theorem0.2}If the equality in \eqref{I}
holds then either $L$ contains a divisor of the movable part of
$|K_X|$, or $L$ is contained in the fixed part of $|K_X|$, or one of
the following cases occurs.
\begin{enumerate}
\item $|K_X|$ is composed of a rational pencil, and the movable part of
$|L|$ is a sum of some fibers of the pencil.

\item $|K_X|$ is composed of a irrational pencil of elliptic curves. The corresponding elliptic fibration
is $f:X\rightarrow C$ with $g(C)\geq2$. There are two line bundles
$A$ and $B$ on $C$ such that $f^*A$ and $f^*B$ are respectively the
movable part of $|L|$ and $|K_X-L|$. The Clifford index of $C$ is
less than $2$. Exactly, we have the following possible cases$:$
\begin{enumerate}
\item $0\leq\chi(\mathcal{O}_X)\leq2$, $C$ is
hyperelliptic and one of $A$ and $B$ is a multiple of $g_2^1;$

\item $\chi(\mathcal{O}_X)=0$, $q=g(C)$, $C$ is a
smooth plane quintic and both of $A$ and $B$ are hyperplane
sections$;$

\item $\chi(\mathcal{O}_S)=0$, $q=g(C)$, $C$ is
trigonal and one of $A$ and $B$ is $g_3^1$.
\end{enumerate}
\end{enumerate}
\end{theorem}

This theorem can be considered as a generalization of the Clifford
theorem. We have another type of generalization as follows.

\begin{theorem}\label{theorem0.1}
Let $X$ be a smooth minimal complex projective surface of general
type. Let $L$ be a special divisor on $X$ such that $L\nsim K_X$,
then $h^0(L)\leq K_XL/2+1$. If the equality holds, then one of the
following cases occurs.
\begin{enumerate}
\item $h^0(L)=1$ and $L$ is a sum of $(-2)$-curves.
\item The movable part of $|L|$ has no base points and $\varphi_L:X\rightarrow \mathbf{P}^1$ is a projective
surjective morphism, whose general fiber is an irreducible smooth
curve of genus $2$.
\item The movable part of $|L|$ has no base points and $\varphi_L$ is
generically $2$ to $1$ onto a surface of minimal degree in
$\mathbf{P}^{h^0(L)-1}$.
\end{enumerate}
\end{theorem}

The two theorems will be proved in Sections \ref{sec:1} and
\ref{sec:2}, respectively.

The organization of the paper is as follows. In Section \ref{sec:1},
we prove Theorem \ref{theorem0.2}. In Section \ref{sec:2}, we will
give some Clifford type inequalities on a surface (Propositions
\ref{proposition1.2} and \ref{proposition1.5}) and prove Theorem
\ref{theorem0.1}.  In Section \ref{sec:3}, we use these two
inequalities to define two indices $\alpha(X)$ and $\beta(X)$ on $X$
like the Clifford index in the case of a curve. We study some basic
properties of $\alpha(X)$ and $\beta(X)$ and give some bounds for
them (Propositions \ref{proposition2.4} and \ref{proposition2.6}).
In Section \ref{sec:4}, we give a detailed description of $X$, when
$\alpha$ and $\beta$ are zero (Theorems \ref{theorem3.1} and
\ref{theorem3.2}). In Section \ref{sec:5}, we use our inequalities
to give some bounds for the number of moduli of surfaces (Theorem
\ref{theorem4.2}).

Throughout the paper, we let $X$ be a smooth complex projective
surface and $K_X$ be its canonical divisor. $p_g$ and $q$ denote,
respectively, $h^0(K_X)$ and $h^1(\mathcal{O}_X)$. For a divisor $L$
on $X$, we let $\varphi_L$ be the rational map defined by the linear
system $|L|$. $|L|$ is said to be composed of a pencil if
$\dim\varphi_L(X)=1$. Numerical equivalence between divisors is
denoted by $\equiv$ and linear equivalence by $\sim$. $g_d^r$
denotes a linear system of degree $d$ and dimension $r$ on a smooth
projective curve. If $E$ is a vector space we will denote by
$\mathbb{P}E$ the space of one-dimension subspaces of $E$.

The author would like to express his appreciation to professor
Sheng-Li Tan for his advice, encouragement and the helpful
discussions. The author is also grateful to the referee for
providing him some valuable suggestions and pointing out grammatical
mistakes.

\section{Proof of Theorem \ref{theorem0.2}}\label{sec:1}
In this section, we will prove Theorem \ref{theorem0.2}. In the
first place, we need the following key lemma.

\begin{lemma}\label{lemma2.1}
Suppose $Z$ is a projective variety. Let $L$ and $D$ be cartier
divisors on $Z$. Assume $Y$ is an irreducible and reduced closed
subscheme of $Z$ and denote $\mathcal{I}_Y$ the ideal sheaf of $Y$
in $Z$. If $h^0(L)-h^0(\mathcal{I}_Y(L))>0$ and
$h^0(D)-h^0(\mathcal{I}_Y(D))>0$, then we have
$$h^0(L)-h^0(\mathcal{I}_Y(L))+h^0(D)-h^0(\mathcal{I}_Y(D))\leq h^0(L+D)-h^0(\mathcal{I}_Y(L+D))+1.$$
\end{lemma}
\begin{proof}
For any Cartier divisor $A$ on $Z$, we have the standard exact
sequence
$$0\rightarrow \mathcal{I}_Y(A)\rightarrow \mathcal{O}_Z(A)\xrightarrow{r_Y}\mathcal{O}_Y(A)\rightarrow 0,$$
where $r_Y$ is the restriction map. We consider the linear system
$r_Y|A|$ on $Y$:
$$r_Y|A|=\mathbb{P}r_Y(H^0(A))\subset \mathbb{P}H^0(\mathcal{O}_Y(A)).$$ We then define a map
\begin{eqnarray*}
\mu:r_Y|L|\times r_Y|D|&\rightarrow& r_Y|L+D|,\\
         (L_1,D_1)     &\mapsto    & L_1+D_1.
\end{eqnarray*}
It is easy to check that $\mu$ is well defined. But every element of
$r_Y|L+D|$ has finite components, thus $\mu$ is finite. Hence
$$\dim(\text{Im}(\mu))=\dim(r_Y|L|\times r_Y|D|)=h^0(L)-h^0(\mathcal{I}_Y(L))-1+h^0(D)-h^0(\mathcal{I}_Y(D))-1.$$
We know that $$h^0(L+D)-h^0(\mathcal{I}_Y(L+D))-1=\dim
r_Y|L+D|\geq\dim(\text{Im}(\mu)).$$ We get our desired inequality.
\end{proof}

\begin{remark}
If we take $Y$ to be an irreducible and reduced divisor, then the
inequality is $h^0(L)-h^0(L-Y)+h^0(D)-h^0(D-Y)\leq
h^0(L+D)-h^0(L+D-Y)+1$. Furthermore, if $Y$ is ample enough, such
that $h^0(L-Y)=h^0(D-Y)=h^0(L+D-Y)=0$, then we get
$h^0(L)+h^0(D)\leq h^0(L+D)+1$ which is well known.
\end{remark}

\begin{proof}[Proof of Theorem $\ref{theorem0.2}$]
If $h^0(K_X-L)=1$ or $h^0(L)=1$, we have $h^0(L)=p_g$ or
$h^0(K_X-L)=p_g$, respectively. Our conclusions are obvious. Hence
we assume $h^0(L)\geq 2$ and $h^0(K_X-L)\geq 2$. In particular, $X$
is either an elliptic surface or a surface of general type.

Let $|L|=|M|+V$ be the decomposition into its movable and fixed
parts. We claim that $h^0(K_X-L)>h^0(K_X-L-M)$. This is because if
$h^0(K_X-L)=h^0(K_X-L-M)$, then
$$h^0(K_X-L)+h^0(M)=h^0(K_X-L-M)+h^0(M)\leq h^0(K_X-L)+1.$$ This implies
$h^0(M)\leq 1$. It is absurd. Hence we proved the claim.

If $\dim\varphi_L(X)=2$, then $h^0(L)\geq 3$ and the general member
of $|M|$ is reduced and irreducible. Since $h^0(K_X-L)>h^0(K_X-L-M)$
and $h^0(L)-h^0(L-M)=h^0(L)-1\geq2$, the conditions of Lemma
\ref{lemma2.1} are satisfied. Hence by Lemma \ref{lemma2.1}, we have
$$h^0(L)-h^0(L-M)+h^0(K_X-L)-h^0(K_X-L-M)\leq p_g+1-h^0(K_X-M).$$
Since $h^0(L)+h^0(K_X-L)=p_g+1$, we get $h^0(K_X-M)\leq
h^0(K_X-L-M)+1$. This implies $h^0(K_X-L-M)\geq h^0(K_X-M)-1\geq1$.
Thus we conclude that
$$h^0(K_X-M)-1+h^0(M)\leq h^0(K_X-L-M)+h^0(M)\leq h^0(K_X-L)+1,$$
i.e., $h^0(K_X-M)+h^0(M)\leq h^0(K_X-L)+2$. Since $h^0(K_X-M)\geq
h^0(K_X-L)$, we obtain $h^0(M)\leq2$. It contradicts that
$h^0(M)=h^0(L)\geq 3$. Therefore $|L|$ is composed of a pencil.
Similarly, $|K_X-L|$ is also composed of a pencil.

Since $h^0(L)+h^0(K_X-L)=p_g+1$, i.e.,
$\dim|K_X|=\dim|L|+\dim|K_X-L|$, we can write every divisor in
$|K_X|$ as a divisor in $|L|$ plus a divisor in $|K_X-L|$. Hence
$|K_X|$ is composed of a pencil. Let $\pi:\widetilde{X}\rightarrow
X$ be a composite of blowing-ups such that the movable part of
$|\pi^* K_X|$ is base point free. We can assume that $\pi$ is the
shortest among those with such a property. Let
$\widetilde{X}\xrightarrow{f}
C\xrightarrow{\varepsilon}\mathbf{P}^{p_g-1}$ be the Stein
factorization of $\varphi_{\pi^* K_X}$. Then there are two base
point free divisors $A$ and $B$ on $C$ such that $f^*A$, $f^*B$ and
$f^*(A+B)$ are respectively the movable part of $\pi^*L$,
$\pi^*(K_X-L)$ and $\pi^*K_X$. Thus $h^0(L)=h^0(\pi^*
L)=h^0(f^*A)=h^0(A)$, $h^0(K_X-L)=h^0(B)$ and $h^0(A+B)=p_g$. If
$g(C)=1$, we have $h^0(A)+h^0(B)=\deg (A+B)=h^0(A+B)$. This implies
$h^0(L)+h^0(K_X-L)=p_g$ which contradicts our assumptions. Hence
$g(C)\neq1$. When $X$ is of general type, we know that $g(C)=0$,
$q\leq2$ by Xiao's estimate in \cite{X} and, therefore, $|K_X|$ is
composed of a rational pencil.

When $X$ is not of general type, it must be an elliptic surface. It
follows that the movable part of $|K_X|$ is base point free,
$\widetilde{X}=X$ and the general fiber of $f$ is an elliptic curve.
If $h^1(A)=0$, then $h^0(A)=\deg A-g(C)+1$. Thus we obtain
$$h^0(B)=h^0(A+B)+1-h^0(A)=\deg B+1+h^1(A+B)\geq \deg B+1.$$ Hence
$g(C)=0$. Similarly, if $h^1(B)=0$, we also have $g(C)=0$. Next we
assume that both of $A$ and $B$ are special divisors and
$g(C)\geq2$. By the Clifford theorem, we have
$$p_g+1=h^0(A)+h^0(B)\leq\frac{\deg(A+B)}{2}+2\leq\frac{\deg f_*\omega_X}{2}+2
=\frac{p_g+g(C)-1}{2}+2.$$ Hence we obtain $p_g\leq g(C)+1\leq q+1$,
i.e., $\chi(\mathcal{O}_X)\leq2$. If $h^0(A)=\deg A/2+1$ or
$h^0(B)=\deg B/2+1$, we get the case $(a)$ immediately. If
$h^0(A)\leq(\deg A+1)/2$ and $h^0(B)\leq(\deg B+1)/2$, then we have
$$p_g+1=h^0(A)+h^0(B)\leq\frac{\deg(A+B)}{2}+1\leq\frac{p_g+g(C)-1}{2}+1.$$
This implies $p_g\leq g(C)-1\leq q-1$, i.e.,
$\chi(\mathcal{O}_X)\leq0$. Therefore we know that
$\chi(\mathcal{O}_X)=0$, $g(C)=q=p_g+1$, $h^0(A)=(\deg A+1)/2$ and
$h^0(B)=(\deg B+1)/2$. By the classical knowledge of algebraic
curves, we get the cases $(b)$ and $(c)$.
\end{proof}

\section{Proof of Theorem \ref{theorem0.1}}\label{sec:2}
In this section, firstly we will give some Clifford type
inequalities. Let $L$ be a divisor on a smooth minimal complex
projective surface $X$ of general type. Let $|L|=|M|+V$ be the
decomposition into its movable and fixed parts, and $W$ the image of
$\varphi_L$.

\begin{proposition}\label{proposition1.2}
If $LK_X\geq 0$, we have
$$h^0(L)\leq \max\Big\{\frac{K_XL}{2}+1, \frac{(K_XL)^2}{2K_X^2}+2\Big\}.$$
\end{proposition}

\begin{proof}
We can first assume that $h^0(L)\geq 3$.

Case A. $\dim W=1$. We can write
$L\thicksim\sum_{i=1}^{a}F_i+V\equiv aF+V$, where $a\geq h^0(L)-1$,
the $F_i's$ are the fibers of $\varphi_L$ and $F^2\geq 0$. Because
of the nefness of $K_X$ we see that
$$LK_X=aFK_X+VK_X\geq(h^0(L)-1)FK_X.$$ This implies $LK_X\geq 2FK_X$. When
$FK_X\geq 2$, we get $h^0(L)\leq LK_X/2+1$.

When $FK_X=1$, we have $F^2K_X^2\leq(FK_X)^2=1$ and $LK_X\geq 2$.
But since $FK_X\equiv F^2(\bmod~2)$, this implies $F^2=1$ and
$K_X^2=1$. Hence
$$h^0(L)\leq LK_X+1\leq \frac{(LK_X)^2}{2}+1=\frac{(LK_X)^2}{2K_X^2}+1.$$

When $FK_X=0$, we get $F^2\leq 0$ by Hodge's index theorem. Thus we
have $F^2=0$ and $F\equiv 0$. Hence $F=0$. It is absurd.

Case B. $\dim W=2$. In this case, we have $$M^2\geq
(\deg\varphi_L)(\deg W)\geq (\deg\varphi_L)(h^0(L)-2).$$ When
$\deg\varphi_L\geq 2$, we obtain $M^2\geq 2h^0(L)-4$. When
$\deg\varphi_L=1$, because $X$ is a surface of general type, $W$ is
not a ruled surface. Hence $\deg W\geq 2n-2=2h^0(L)-4$. This implies
$M^2\geq 2h^0(L)-4$. We obtain
$$LK_X=MK_X+VK_X\geq MK_X\geq\sqrt{M^2K_X^2}\geq\sqrt{(2h^0(L)-4)K_X^2}.$$
Therefore we conclude that $h^0(L)\leq (K_XL)^2/2K_X^2+2$.
\end{proof}

The following Castelnuovo type inequality is standard (cf.
\cite[Lemma 2.1]{Ko1}).
\begin{lemma}\label{lemma1.3}
Let $S$ be a smooth projective surface, $D$ a divisor on $S$ such
that $|D|$ defines a birational map of $S$ onto the image. If $|D|$
has no fixed part and $(K_S-D)D\geq 0$, then $D^2\geq 3h^0(D)-7$.
\end{lemma}

\begin{proposition}\label{proposition1.5}
If $K_XL\geq K_X^2$, then $h^0(L)\leq (K_XL)^2/2K_X^2+2$. If $0\leq
K_XL\leq K_X^2$, then $h^0(L)\leq K_XL/2+2$. If one of the
conditions holds, then $\varphi_L$ is generically $2$ to $1$ onto a
surface of minimal degree in $\mathbf{P}^{h^0(L)-1}$.
\end{proposition}

\begin{proof}
Case 1. $K_XL\geq K_X^2$. This implies $(K_XL)^2/2K_X^2+2\geq
K_XL/2+2$. By Proposition \ref{proposition1.2}, we have $h^0(L)\leq
(K_XL)^2/2K_X^2+2$. When the equality holds, from the proof of
Proposition \ref{proposition1.2}, we obtain $\dim W=2$,
$M^2=2h^0(L)-4$, $(MK_X)^2=M^2K_X^2$ and $VK_X=0$. Hence $|M|$ is
base point free, $V$ is a sum of some $(-2)$-curves and $M\equiv
rK_X$ for some rational number $r$.

Assume $\deg\varphi_L=1$ and $h^0(M)\geq 4$. Then by Lemma
\ref{lemma1.3}, we have $2h^0(M)-4=M^2\geq 3h^0(M)-7$, i.e.,
$h^0(M)\leq 3$. This is a contradiction.

Assume $\deg\varphi_L=1$ and $h^0(M)\leq 3$. Then since $\dim W=2$,
we have $h^0(M)=3$ and $W=\mathbf{P}^2$. Hence $X$ is a rational
surface. It contradicts our assumption on $X$.

Therefore $\deg\varphi_L=2$ and $\deg W=h^0(L)-2$. Thus $W$ is a
surface of minimal degree.

Case 2. $K_XL\leq K_X^2$. In this case we have
$(K_XL)^2/2{K_X}^2+2\leq K_XL/2+2$. By Proposition
\ref{proposition1.2}, we obtain $h^0(L)\leq K_XL/2+2$. When the
equality holds, we also have $\dim W=2$. Therefore
$(K_XL)^2/2{K_X}^2+2=K_XL/2+2$, i.e., $K_XL= K_X^2$. Thus we can
finish our proof similarly as Case 1.
\end{proof}

Now we will prove Theorem \ref{theorem0.1}.
\begin{proof}[Proof of Theorem $\ref{theorem0.1}$]
Since $h^0(K_X-L)=h^2(L)>0$, we have $(K_X-L)K_X\geq 0$. By
Proposition \ref{proposition1.5}, we get $h^0(L)\leq K_XL/2+2$.

If $h^0(L)=K_XL/2+2$, we have $K_X^2=M^2=2h^0(L)-4$ and
$(MK_X)^2=M^2K_X^2$. Therefore $M\equiv K_X$. But since
$h^0(K_X-M)\geq h^0(K_X-L)>0$, we know that $M\sim K_X$. Hence
$h^0(-V)=h^0(M-L)=h^0(K_X-L)>0$. This implies $V=0$ and $L=M\sim
K_X$. It contradicts the assumption $L\nsim K_X$. Therefore we
obtain $h^0(L)\leq (K_XL-1)/2+2$.

If $h^0(L)=(K_XL-1)/2+2$, we have $K_XL=2h^0(L)-3$. When $\dim W=1$,
we obtain $$K_XL=2h^0(L)-3\geq (h^0(L)-1)FK_X.$$ This implies
$FK_X=1$. Since $F^2{K_X}^2\leq(FK_X)^2=1$ and $FK_X\equiv
F^2\pmod2$, we have $F^2=K_X^2=FK_X=1$. Thus $F^2K_X^2=(FK_X)^2=1$.
This implies $F\equiv K_X$. Since $h^0(K_X-F)\geq h^0(K_X-L)>0$, we
know that $F\sim K_X$. Hence $L\sim K_X$. It also contradicts the
assumption $L\nsim K_X$. When $\dim W=2$, we have $M^2\geq
2h^0(L)-4=K_XL-1\geq K_XM-1$. Since $M^2\equiv MK_X\pmod2$, we get
$M^2\geq K_XM$. Because $\dim W=2$, we can find a reduced and
irreducible curve in $|M|$. Hence $M$ is a nef divisor. Since
$h^0(K_X-M)\geq h^0(K_X-L)>0$, we have $(K_X-M)M\geq 0$ and
$(K_X-M)K_X\geq 0$. It follows that $M^2\leq K_XM\leq K_X^2$. Hence
$M^2=K_XM\leq K_X^2$. By Hodge's index theorem, we get $M^2K_X^2\leq
(K_XM)^2=(M^2)^2$, i.e., $K_X^2\leq M^2$. Therefore $K_X^2=M^2=K_XM$
and $M^2K_X^2=(K_XM)^2$. Thus $M\equiv K_X$. Because $h^0(K_X-M)>0$
and $M\sim K_X$, we know that $M\sim K_X\sim L$. It contradicts the
assumption $L\nsim K_X$ again. Hence we conclude that
$$h^0(L)\leq\frac{K_XL-2}{2}+2=\frac{K_XL}{2}+1.$$

Now we assume the equality holds, i.e., $K_XL=2h^0(L)-2$. If
$h^0(L)=1$, then $K_XL=0$. Hence $L$ is a sum of $(-2)$-curves.

When $\dim W=1$, we have
$$2h^0(L)-2=LK_X=aFK_X+VK_X\geq(h^0(L)-1)FK_X.$$
This implies $K_XF\leq 2$. If $K_XF=1$, by Hodge's index theorem, we
have $F^2K_X^2\leq (FK_X)^2=1$. This implies $F^2=K_X^2=1$. But
$K_X^2\geq K_XL=2h^0(L)-2\geq2$. It is impossible. Hence we have
$K_XF=2$. It follows that $F^2K_X^2\leq (FK_X)^2=4$. Since
$K_X^2\geq 2$ and $FK_X\equiv F^2\pmod2$, we obtain $F^2=0$ or
$F^2=2$.

If $F^2=2$, then $K_X^2=K_XL=2$. Thus $F^2K_X^2=(K_XF)^2=4$. By
Hodge's index theorem, we know that $F\equiv K_X$. This implies
$V\sim0$ and $K_X\sim L\sim F$. It contradicts the assumption
$L\nsim K_X$.

If $F^2=0$, then the movable part of $|L|$ is base point free. Since
$K_XF=2$, we conclude that $a=h^0(L)-1$, $W\cong\mathbf{P}^1$ and
$g(F)=(F^2+FK_X)/2+1=2$. Therefore, the general fiber of $\varphi_L:
X\rightarrow W\cong\mathbf{P}^1$ is an irreducible smooth curve of
genus $2$.

When $\dim W=2$, we have $h^0(L)\geq 3$ and $K_XL=2h^0(L)-2\geq4$.
Since $M^2\geq 2h^0(L)-4=K_XL-2\geq K_XM-2$, $K_XM\geq M^2$ and
$M^2-K_XM$ is even, we know that $M^2=K_XM$ or $M^2=K_XM-2$.

If $M^2=K_XM$, the inequality $(K_XM)^2\geq K_X^2M^2$ implies that
$M^2\geq K_X^2$. Since $K_X^2\geq K_XM=M^2$, we have
$K_X^2=M^2=K_XM$. By Hodge's index theorem, we obtain $M\equiv K_X$.
Since $h^0(K_X-M)>0$, we obtain $L\sim M\sim K_X$. It contradicts
the assumption $L\nsim K_X$.

If $M^2=K_XM-2=2h^0(M)-4$, we have that $|M|$ is base point free and
$\varphi_L$ is generically $2$ to $1$ onto a surface of minimal
degree in $\mathbf{P}^{h^0(L)-1}$.
\end{proof}

\section{Clifford type indices on a surface}\label{sec:3}
For a smooth connected projective curve, we have an invariant, the
Clifford index, introduced by Martens \cite{Ma}. It plays an
important role in the study of curves. Because of Theorems
\ref{theorem0.1} and \ref{theorem0.2}, we can define two indices of
Clifford type on a smooth minimal surface X of general type.
\begin{definition}
For a divisor $L$ on $X$, we define two indices $\alpha(L)$ and
$\beta(L)$ by
\begin{eqnarray*}
\alpha(L)&=&K_XL-2h^0(L)+2,\\
\beta(L) &=&q+\frac{1}{2}L(K_X-L)-h^1(L).
\end{eqnarray*}
\end{definition}

Note that by the Serre duality theorem, we have
$\beta(L)=\beta(K_X-L)$ and by the Reimann-Roch theorem, we have
$h^0(L)+h^0(K_X-L)=1+p_g-\beta(L)$ and
\begin{eqnarray*}
\alpha(L)+\alpha(K_X-L)&=&K_X^2-2(h^0(L)+h^0(K_X-L))+4 \\
                       &=&K_X^2-2(1+p_g-\beta(L))+4\\
                       &=&K_X^2-2p_g+2\beta(L)+2.
\end{eqnarray*}

Next we define indices $\alpha(X)$ and $\beta(X)$ for the surface
$X$.
\begin{definition}
Let $\mathcal{S}=\{L\in\Pic(X);h^0(L)\geq 2, h^0(K_X-L)\geq 2\}$, we
define $\alpha(X)$ and $\beta(X)$ by
\begin{eqnarray*}
\alpha(X)=
\begin{cases}\underset{L\in\mathcal{S}}{\min}~ \alpha(L) & \mathcal{S}\neq
\emptyset \\ \infty & \mathcal{S}=\emptyset
\end{cases}\\
\beta(X)=
\begin{cases}\underset{L\in\mathcal{S}}{\min}~ \beta(L) & \mathcal{S}\neq
\emptyset\\ \infty & \mathcal{S}=\emptyset.
\end{cases}
\end{eqnarray*}

\end{definition}

Similarly as in the curve case, we say that $L$ computes the index
$\alpha(X)$ or $\beta(X)$, if $\alpha(X)=\alpha(L)$ or
$\beta(X)=\beta(L)$, respectively.

\begin{remark}\label{remark2.3}
When $L$ computes $\alpha(X)$ or $\beta(X)$, we can always assume
$|L|$ has no fixed part. This assumption is convenient for our work.
The reason is as follows. Let $|L|=|M|+V$ be the decomposition into
its movable and fixed parts. If $L$ computes $\alpha(X)$, we have
$h^0(L)=h^0(M)$ and $VK_X\geq 0$. Therefore $K_XL-2h^0(X,L)+2\geq
K_XM-2h^0(X,M)+2$, i.e., $\alpha(L)\geq\alpha(M)$. If $L$ computes
$\beta(X)$, we have $h^0(L)+h^0(K_X-L)=1+p_g-\beta(X)$. But since
$h^0(L)+h^0(K_X-L)\leq h^0(M)+h^0(K_X-M)\leq 1+p_g-\beta(X)$, we
have $h^0(M)+h^0(K_X-M)=1+p_g-\beta(X)$. Hence $M$ computes
$\beta(X)$ too.
\end{remark}

\begin{example}
Let $S_d$ be a generic hypersurface of degree $d$ in $\mathbf{P}^3$.
$H$ denote the hyperplane section of $S_d$. When $d\geq5$, $S_d$ is
a minimal surface of general type and $K_{S_d}=(d-4)H$. In this
case, by the Noether-Lefschetz theorem, we have $\Pic(S_d)\cong
\mathbb{Z}H$. Hence $\alpha(S_5)=\beta(S_5)=\infty$.

Now we assume $d\geq6$. Let $n$ be an integer such that $1\leq n\leq
d-5$. Then we have $$h^0(nH)=\frac{1}{6}(n+1)(n+2)(n+3).$$ Thus we
obtain
\begin{eqnarray*}
\alpha(nH)&=&nHK_{S_d}-2h^0(nH)+2\\
          &=&nd(d-4)-\frac{1}{3}(n+1)(n+2)(n+3)+2.
\end{eqnarray*}
Hence $\alpha(S_d)={\min}_{1\leq n\leq
d-5}\alpha(nH)=\alpha(H)=d(d-4)-6$. We also have
\begin{eqnarray*}
\beta(nH)&=&p_g(S_d)+1-h^0(nH)-h^0((d-4-n)H)\\
         &=&-\frac{1}{2}d(n^2-(d-4)n).
\end{eqnarray*}
Therefore $\beta(S_d)={\min}_{1\leq n\leq
d-5}\beta(nH)=\beta(H)=d(d-5)/2$.
\end{example}

For surfaces with $\alpha=\infty$, we have the following theorem.
\begin{theorem}
If $S$ is a surface with $\alpha(S)=\infty$, then
$\alpha(S')=\infty$ for every small deformation $S'$ of $S$.
\end{theorem}
\begin{proof}
Let $f:\mathcal{X}\rightarrow \Delta$ be a small deformation of
$\mathcal{X}_0=S$, $0\in \Delta$, such that the Picard scheme
$\Pic_{\mathcal{X}/\Delta}$ and the Poincar\'e line bundle
$\mathcal{L}$ on $\mathcal{X}\times\Pic_{\mathcal{X}/\Delta}$ exist
(cf. \cite{K}). Put

$$W_{m,n}=\{y\in\Pic_{\mathcal{X}/\Delta}~;h^0(\mathcal{L}_y)\geq m,
h^2(\mathcal{L}_y)\geq n\}.$$ By the semicontinuity theorem
\cite[Theorem 12.8]{Ha}, we know that $W_{m,n}$ is a closed
subscheme of $\Pic_{\mathcal{X}/\Delta}$. Consider the natural
morphism $\pi:\Pic_{\mathcal{X}/\Delta}\rightarrow \Delta$. Then
$\{p\in \Delta~;W_{m,n}\cap\pi^{-1}(p)=\emptyset\}$ is an open
subset of $\Delta$. Since $\alpha(S)=\infty$, we have $\{L\in
\Pic(S)~;h^0(L)\geq 2, h^2(L)\geq 2\}=\emptyset$. Hence
$W_{2,2}\cap\pi^{-1}(0)=\emptyset$ and $\{p\in
\Delta~;W_{2,2}\cap\pi^{-1}(p)=\emptyset\}\neq\emptyset$. Thus for
every $p\in\{p\in \Delta~;W_{2,2}\cap\pi^{-1}(p)=\emptyset\}$, we
have $\alpha(f^{-1}(p))=\infty$. This completes the proof of the
theorem.
\end{proof}

The above theorem tells us the surfaces with $\alpha=\infty$ form an
open subset of the moduli of surfaces. We now give some bounds for
$\alpha(X)$ and $\beta(X)$ as follows:

\begin{proposition}\label{proposition2.4}
If $\alpha(X)\neq\infty$, then $0\leq\alpha(X)\leq
K_X^2-2\chi(\mathcal{O}_X)+6$.
\end{proposition}
\begin{proof}
$\alpha(X)\geq 0$ is an easy consequence of Theorem
\ref{theorem0.1}. Suppose that $L$ computes $\alpha(X)$. Then we
have $\alpha(X)= K_XL-2h^0(X,L)+2$. By Remark \ref{remark2.3}, we
can assume $|L|$ has no fixed part. Let $W$ be the image of
$\varphi_L$.

Case A. $\dim W=1$. In this case, we have
$L\thicksim\sum_{i=1}^{a}F_i\equiv aF$, where the $F_i's$ are the
fibers of $\varphi_L$. Since
$$a\geq
h^0(L)-1=\frac{K_XL-\alpha(X)}{2}=\frac{aK_XF-\alpha(X)}{2},$$ we
have
\begin{equation}\label{1}
2a+\alpha(X)\geq aK_XF.
\end{equation}
By the Riemann-Roch theorem, we obtain
\begin{eqnarray}\label{A1}
h^0(L)+h^0(K_X-L)\geq\frac{a^2}{2}F^2-\frac{a}{2}K_XF+\chi(\mathcal{O}_X).
\end{eqnarray}
Since $h^0(L)=(K_XL-\alpha(X))/2+1$ and $h^0(K_X-L)\leq
(K_X(K_X-L)-\alpha(X))/2+1$, we get from (\ref{A1}) the inequality
\begin{eqnarray}\label{2}
K_X^2-2\chi(\mathcal{O}_X)+4+aK_XF-a^2F^2\geq 2\alpha(X).
\end{eqnarray}

If $F^2\geq 1$, we have $2a\leq a^2+1\leq a^2F^2+1$. This and
(\ref{1}) imply $aK_XF-a^2F^2\leq\alpha(X)+1$. Hence by (\ref{2}),
we get $\alpha(X)\leq K_X^2-2\chi(\mathcal{O}_X)+5$.

If $F^2=0$, then $|L|$ has no base points. We can assume that $F$ is
a smooth and irreducible curve. When $h^0(L)=2$, we have
$W=\mathbf{P}^1$, and $\alpha(X)+2=aK_XF$. By (\ref{2}), we get our
conclusion immediately. When $h^0(L)\geq 3$, from the standard exact
sequence
$$0\rightarrow\mathcal{O}_X(L-F)\rightarrow\mathcal{O}_X(L)
\rightarrow\mathcal{O}_{F}\rightarrow 0,$$ it follows that
$h^0(L-F)\geq h^0(L)-1\geq 2$. Therefore
$$\frac{1}{2}((a-1)K_XF-\alpha(X))+1\geq h^0(L-F)\geq h^0(L)-1=\frac{1}{2}(aK_XF-\alpha(X)).$$
This implies $K_XF\leq 2$. Since $F^2=0$, we have $K_XF=2$. Hence
$h^0(L-F)= h^0(L)-1=a-\alpha(X)/2$, for a general fiber $F$.
Inductively, we can get $h^0(L-iF)=h^0(L)-i$, for $1\leq i\leq
h^0(L)-1$. Let $k=h^0(L)-2$, then $h^0(L-kF)=2$. On one hand, by the
Riemann-Roch theorem, we obtain
\begin{eqnarray}\label{A2}
h^0(L-kF)+h^0(K_X-L+kF) & \geq & \frac{1}{2}(L-kF)(L-kF-K_X)+\chi(\mathcal{O}_X)\nonumber\\
                        &   =  & k-a+\chi(\mathcal{O}_X).
\end{eqnarray}
On the other hand,
\begin{eqnarray*}
h^0(L-kF)+h^0(K_X-L+kF) & \leq & 2+\frac{1}{2}K_X(K_X-L+kF)-\frac{1}{2}\alpha(X)+1\\
                        &   =  & \frac{1}{2}K_X^2+k-a-\frac{1}{2}\alpha(X)+3.
\end{eqnarray*}
Combining these two inequalities, we can get $\alpha(X)\leq
K_X^2-2\chi(\mathcal{O}_X)+6$.

Case B. $\dim W=2$. This case implies that $L^2\geq
2h^0(L)-4=K_XL-\alpha(X)-2$. Hence by the Riemann-Roch theorem, we
obtain
\begin{eqnarray}\label{B}
h^0(L)+h^0(K_X-L) & \geq & \frac{1}{2}L^2-\frac{1}{2}K_XL+\chi(\mathcal{O}_X)\nonumber\\
                  & \geq & -\frac{1}{2}\alpha(X)-1+\chi(\mathcal{O}_X).
\end{eqnarray}
Since
\begin{eqnarray*}
h^0(L)+h^0(K_X-L) & \leq & \frac{1}{2}(K_XL-\alpha(X))+1+\frac{1}{2}(K_X(K_X-L)-\alpha(X))+1\\
                  &   =  & \frac{1}{2}K_X^2-\alpha(X)+2,
\end{eqnarray*}
we obtain $K_X^2/2-\alpha(X)+2\geq
-\alpha(X)/2-1+\chi(\mathcal{O}_X)$, i.e., $\alpha(X)\leq
K_X^2-2\chi(\mathcal{O}_X)+6$.
\end{proof}

We can see easily the following corollary.

\begin{corollary}\label{corollary2.5}If
$\alpha(X)=K_X^2-2\chi(\mathcal{O}_X)+6$, $L$ computes $\alpha(X)$
and $|L|$ has no fixed part, then $|L|$ is base point free and one
of the following cases occurs.
\begin{enumerate}
\item $h^0(L)=2$, $h^1(L)=0$ and $K_X-L$ also computes $\alpha(X)$.
\item $h^0(L)\geq3$ and $|L|$ is composed of a pencil of genus $2$.
\item $h^1(L)=0$, $K_X-L$ also computes $\alpha(X)$ and $\varphi_L$ is
generically $2$ to $1$ onto a surface of minimal degree in
$\mathbf{P}^{h^0(L)-1}$.
\end{enumerate}
\end{corollary}

\begin{proposition}\label{proposition2.6}
If $\alpha(X)\neq\infty$, then $0\leq\beta(X)\leq \alpha(X)/2+q+1$.
\end{proposition}
\begin{proof}
$\beta(X)\geq 0$ is an easy consequence of Theorem \ref{theorem0.2}.
The following proof is similar to that of Proposition
\ref{proposition2.4}. Keep the notation as in the proof of
Proposition \ref{proposition2.4}. We assume $L$ computes $\alpha(X)$
and $|L|$ has no fixed part. Then we have $\alpha(X)=
K_XL-2h^0(L)+2$ and
\begin{equation}\label{beta}
h^0(L)+h^0(K_X-L)=1+p_g-\beta(L)\leq 1+p_g-\beta(X).
\end{equation}

Case A. $\dim W=1$. By (\ref{A1}) and (\ref{beta}), we obtain
\begin{equation}\label{A3}
\beta(X)\leq q+\frac{a}{2}K_XF-\frac{a^2}{2}F^2.
\end{equation}

If $F^2\geq 1$, we have $2a\leq a^2+1\leq a^2F^2+1$. This and
(\ref{1}) imply $aK_XF-a^2F^2\leq\alpha(X)+1$. From (\ref{A3}), it
follows that $\beta(X)\leq \alpha(X)/2+q+1/2$.

If $F^2=0$, then $|L|$ has no base points. When $h^0(L)=2$, then
$\alpha(X)+2=aK_XF$. It follows from (\ref{A3}) that $\beta(X)\leq
\alpha(X)/2+q+1$. When $h^0(L)\geq 3$, similarly as in the proof of
Proposition \ref{proposition2.4}, we have $K_XF=2$. Hence by
(\ref{A2}), we get
$$h^0(L-kF)+h^0(K_X-L+kF)\geq k-a+\chi(\mathcal{O}_X),$$ where
$k=h^0(L)-2=(K_XL-\alpha(X))/2-1$. On the other hand,
$h^0(L-kF)+h^0(K_X-L+kF)\leq 1+p_g-\beta(X)$. Combining them, we
obtain
\begin{eqnarray*}
\beta(X)\leq q+a-k &  = & q+a-\frac{1}{2}(K_XL-\alpha(X))+1 \\
                   &  = & q+\frac{1}{2}\alpha(X)+1+a-\frac{aK_XF}{2}\\
                   &  = & q+\frac{1}{2}\alpha(X)+1.
\end{eqnarray*}

Case B. $\dim W=2$. By (\ref{B}) and (\ref{beta}), we get
$\beta(X)\leq \alpha(X)/2+q+1$ immediately.
\end{proof}

Similarly as in the case of Corollary \ref{corollary2.5}, we have
the following corollary.
\begin{corollary}
If $\beta(X)=\alpha(X)/2+q+1$, $L$ computes $\alpha(X)$ and $|L|$
has no fixed part, then $|L|$ is base point free and one of the
following cases occurs.
\begin{enumerate}
\item $h^0(L)=2$, $h^1(L)=0$ and $L$ computes $\beta(X)$.
\item $h^0(L)\geq3$ and $|L|$ is composed of a pencil of genus $2$.
\item $h^1(L)=0$, $L$ computes $\beta(X)$ and $\varphi_L$ is
generically $2$ to $1$ onto a surface of minimal degree in
$\mathbf{P}^{h^0(L)-1}$.
\end{enumerate}
\end{corollary}

\section{Surfaces with $\alpha=0$ or $\beta=0$}\label{sec:4}

It is natural to ask what will happen when these indices $\alpha$
and $\beta$ are small. The answers for $\alpha=0$ and $\beta=0$,
respectively, are given in the following theorems. We always assume
$L$ computes $\alpha(X)$ and $|L|$ has no fixed part.
\begin{theorem}\label{theorem3.1}
If $\alpha(X)=0$, then $|L|$ has no base point and one of the
following occurs.

$1$. There exists a projective surjective morphism
$f:X\rightarrow\mathbf{P}^1$, whose general fiber is an irreducible
smooth curve of genus $2$.

$2$. $X$ is the minimal resolution of a double covering of
$\mathbf{P}^2$, whose branch locus is a reduced curve of degree $10$
with only one infinitely near triple point as its essential
singularity. In this case, $K_X^2=7$, $p_g=5$, $q=0$ and $K_X\sim
2L-Z$, where $Z$ is an effective divisor with $LZ=0$ and
$K_XL=2L^2=4$.

$3$. $X$ is the smooth minimal model of a double covering of
$\Sigma_2$, whose branch locus is a reduced curve of
$|8\Delta_0+14\Gamma|$ with at worst negligible singularities. In
this case, $K_X^2=9$, $p_g=6$, $q=0$ and $K_X\sim 3D$, where $2D=L$.

$4$. $X$ is the minimal resolution of a double covering of
$\mathbf{P}^2$, whose branch locus is a reduced curve of degree $10$
with at worst negligible singularities. In this case, $K_X^2=8$,
$p_g=6$, $q=0$ and $K_X\sim 2L$.

\end{theorem}

\begin{proof}
Let $W$ be the image of $\varphi_L$. Since $L$ computes $\alpha(X)$,
we get $K_XL-2h^0(L)+2=\alpha(X)=0$.

When $\dim W=1$, by Theorem \ref{theorem0.1}, we have $|L|$ is base
point free and the general fiber of $\varphi_L: X\rightarrow
W\cong\mathbf{P}^1$ is an irreducible smooth curve of genus 2. Thus
$X$ is the surface of type 1 in the theorem. When $\dim W=2$, by
Theorem \ref{theorem0.1}, we know that $|L|$ is base point free,
$\varphi_L:X\rightarrow W$ is generically $2$ to $1$ and
\begin{equation}\label{*}
L^2=K_XL-2=2h^0(L)-4\geq 2.
\end{equation}
By Hodge's index theorem, we obtain
$$K_X^2L^2\leq (K_XL)^2=(L^2+2)^2=(L^2)^2+4L^2+4.$$ This implies
\begin{equation}\label{4}
K_X^2\leq L^2+\frac{4}{L^2}+4.
\end{equation}
Since $2\leq h^0(K_X-L)\leq K_X(K_X-L)/2+1$, we have
\begin{equation}\label{5}
K_X^2\geq K_XL+2=L^2+4.
\end{equation}
Combining (\ref{4}) and (\ref{5}), we obtain
$$L^2+4\leq K_X^2\leq L^2+\frac{4}{L^2}+4\leq L^2+6.$$
Thus we get three possible cases A: $K_X^2=L^2+4$, B: $K_X^2=L^2+5$
and C: $K_X^2=L^2+6$.

Case A. $K_X^2=L^2+4$. This implies that $K_X^2=K_XL+2$. Then
$$2\leq h^0(K_X-L)\leq\frac{1}{2}K_X(K_X-L)+1=\frac{1}{2}({K_X}^2-K_XL)+1=2.$$
Thus $h^0(K_X-L)=2$. Let $|K_X-L|=|M'|+V'$ be the decomposition into
its movable and fixed parts. Let $\phi:X\dashrightarrow
\mathbf{P}^1$ be the rational map defined by $|K_X-L|$. Then there
exists an irreducible reduced curve $F'$, such that ${F'}^2\geq 0$,
$M'\equiv bF'$ and $b\geq h^0(K_X-L)-1=1$. Since
$$2=(K_X-L)K_X=M'K_X+V'K_X\geq bF'K_X\geq F'K_X,$$
we can get $F'K_X=2$. Hence $b=1$, ${F'}^2=0$ or $2$. If ${F'}^2=2$,
then $(K_X-L)F'=M'F'+V'F'\geq {F'}^2=2$. Thus $LF'\leq K_XF'-2=0$.
By Hodge's index theorem, we get ${F'}^2\leq 0$. It is impossible.
It follows that ${F'}^2=0$, and $|M'|$ is base point free. Hence
$g(F')=({F'}^2+K_XF')/2+1=2$ and $M'\sim F'$. We know that the
general fiber of $\phi:X\rightarrow\mathbf{P}^1$ is an irreducible
smooth curve of genus 2. Therefore $X$ is the surface of type 1 in
the theorem.

Case B. $K_X^2=L^2+5$. Since $2\leq h^0(K_X-L)\leq
K_X(K_X-L)/2+1=5/2$, we get $h^0(K_X-L)=2$. By (\ref{4}), we have
$L^2+5\leq L^2+4/L^2+4$. This implies $2\leq L^2\leq 4$. Since
$L^2=2h^0(L)-4$ is an even number, there are two cases B-I: $L^2=2$
and B-II: $L^2=4$.

Case B-I. We have $K_X^2=7$, $K_XL=4$ and $h^0(L)=3$. By Theorem
\ref{theorem0.2}, we have $p_g(X)=h^0(K_X)\geq h^0(L)+h^0(K_X-L)=5$.
Using Noether's inequality, we know that $7=K_X^2\geq 2p_g(X)-4$,
i.e., $p_g(X)\leq 5$. Thus we get $p_g(X)=5$ and
$K_X^2=7<10=2p_g(X)$. Since $K_X^2\geq 2p_g(X)$, when $X$ is
irregular (See \cite{De}), we conclude that $q(X)=0$.

Since $h^0(L)=3$, we know that $\varphi_L:X\rightarrow
W=\mathbf{P}^2$ is generically $2$ to $1$. Let $X\rightarrow
X'\xrightarrow{f}\mathbf{P}^2$ be the Stein factorization of
$\varphi_L$, $\widetilde{X}$ the canonical resolution of the double
covering and $m_i$ the multiplicity of the corresponding
singularity. $R$ and $B$ denote, respectively, the ramification
divisor and the branch locus of $\varphi_L$. If $H$ denotes a line
on $\mathbf{P}^2$, then we have $K_X=\varphi_L^*(-3H)+R=-3L+R$. By
the theory of double covering (See \cite[\S2]{Hor}, \cite[III,
\S2]{Ho} or \cite[\S1.3]{Xiao}), there exists an effective divisor
$Z$ on $X$, such that $2R=\varphi_L^*B-2Z$ and $LZ=0$. Thus
$$BH=\frac{1}{2}\varphi_L^*B\varphi_L^*H=(R+Z)L=RL=(K_X+3L)L=K_XL+3L^2=10.$$
Hence $B\sim 10H$ and $K_X\sim 2L-Z$. Now we can compute the
invariants of $\widetilde{X}$. We have
\begin{eqnarray*}
\chi(\mathcal{O}_X)=\chi(\mathcal{O}_{\widetilde{X}}) & = & \frac{1}{4}B\left(K_{\mathbf{P}^2}+\frac{1}{2}B\right)+2\chi(\mathcal{O}_{\mathbf{P}^2})-\sum_i\frac{1}{2}\left[\frac{m_i}{2}\right]\left(\left[\frac{m_i}{2}\right]-1\right)\\
                                                      & = & 7-\frac{1}{2}\sum_i\left[\frac{m_i}{2}\right]\left(\left[\frac{m_i}{2}\right]-1\right),
\end{eqnarray*}
\begin{eqnarray*}
K_{\widetilde{X}}^2=2\left(K_{\mathbf{P}^2}+\frac{1}{2}B\right)^2-\sum_i2\left(\left[\frac{m_i}{2}\right]-1\right)^2=8-2\sum_i\left(\left[\frac{m_i}{2}\right]-1\right)^2.
\end{eqnarray*}
From the equality $q(\widetilde{X})=q(X)=0$, it follows that
$$p_g(X)=6-\frac{1}{2}\sum_i\left[\frac{m_i}{2}\right]\left(\left[\frac{m_i}{2}\right]-1\right).$$
Since $p_g(X)=5$, we have $[m_i/2]=2$ for only one index and
$K_{\widetilde{X}}^2=6$. It follows that $\widetilde{X}$ has a
$(-1)$-curve. Therefore $X$ is the surface of type 2 in the theorem.

Case B-II. We have $K_X^2=L^2+5=9$ and $K_XL=L^2+2=6$ by (\ref{*}).
Thus $L^2K_X^2=36=(K_XL)^2$. Using Hodge's index theorem, we have
$L\equiv (2/3)K_X$. It follows from (\ref{*}) that $h^0(L)=4$. By
Theorem \ref{theorem0.2}, we have $p_g(X)=h^0(K_X)\geq
h^0(L)+h^0(K_X-L)=6$. By Noether's inequality, we obtain
$9=K_X^2\geq 2p_g(X)-4$, i.e., $p_g(X)\leq 6$. Hence $p_g(X)=6$ and
$K_X^2=9<12=2p_g(X)$. It follows that $q(X)=0$.

Since $\deg W=L^2/2=2$, either
$W\cong\mathbf{P}^1\times\mathbf{P}^1$ or $W$ is a quadric cone.

Assume $W\cong\mathbf{P}^1\times\mathbf{P}^1$. Two rulings of $W$
allow us to write $L\sim D_1+D_2$ with divisors $D_i$ satisfying
$D_i^2=0$ $(i=1,2)$. Since $6=K_XL=K_XD_1+K_XD_2$, we may assume
that $K_XD_1$ is an even integer not greater than 3. Hence
$K_XD_1=2$. But, this is absurd, because $LD_1=(2/3)K_XD_1=4/3$.

Now we assume that $W$ is a quadric cone. In this case, by the same
argument as in the proof of \cite[Lemma 2, Case II b]{Hor}, we have
$L\sim 2D+G$, where $|D|$ is a pencil and $G$ is an effective
divisor with $LG=0$. From the equality $4=L^2=L(2D+G)$, it follows
that $LD=2$. Since $L\equiv(2/3)K_X$, we have $K_XD=3$. From $LD=2$,
we get $(2D+G)D=2D^2+DG=2$, hence $D^2=0$ or 1. But $3=K_XD\equiv
D^2\pmod 2$, hence $D^2=1$ and $DG=0$. The equality $0=LG=2DG+G^2$
implies that $G^2=0$. Then by Hodge's index theorem, we have $G=0$.
Thus $L\sim 2D$ and $K_X\equiv 3D$.

Since $h^0(D)\leq(1/2)K_XD+1=5/2$, we have $h^0(D)=2$. By $D^2=1$,
we know $|D|$ has one base point $P$. Let
$\sigma:\widehat{X}\rightarrow X$ be the blowing-up with center $P$
and put $E=\sigma^{-1}(P)$. Then the movable part $\widehat{D}$ of
$|\sigma^*D|$ defines a holomorphic map
$g:\widehat{X}\rightarrow\mathbf{P}^1$. Since $|L|$ has no base
point, there exists $\eta\in
H^0(\widehat{X},\mathcal{O}_{\widehat{X}}(\sigma^*L))$ which does
not vanish on $E$. Take $\xi\in
H^0(\widehat{X},\mathcal{O}_{\widehat{X}}(2E))$ such that
$(\xi)=2E$. Then $\xi/\eta$ is a meromorphic section of
$\mathcal{O}_{\widehat{X}}(-2\widehat{D})$. Then $g$ and $\xi/\eta$
define a rational map $h:\widehat{X}\rightarrow \Sigma_2$. Since
$\eta$ does not vanish on $E$, $h$ is defined everywhere so that
$h^*\Delta_0=2E$. We consider the linear system $|\Delta_0+2\Gamma|$
on $\Sigma_2$. This give rise to a morphism
$q:\Sigma_2\rightarrow\mathbf{P}^3$ whose image coincides with $W$
up to an automorphism of $\mathbf{P}^3$. Then by the construction,
we have the following commutative diagram.
\[\begin{CD}
\widehat{X} @>h>>      \Sigma_2 \\
@V\sigma VV            @VVqV \\
X           @>\varphi_L>> W
\end{CD}\]
Let $H$ be a plane on $\mathbf{P}^3$. $R$ and $B$ denote,
respectively, the ramification divisor and the branch locus of $h$.
Then there exists an effective divisor $Z$ on $\widehat{X}$, such
that $2R=h^*B-2Z$ and $Z$ is contracted by $h$. Since
$q^*H=\Delta_0+2\Gamma=-(1/2)K_{\Sigma_2}$, we have
$\sigma^*(2D)=\sigma^*L=\sigma^*\varphi_L^*H=h^*q^*H=h^*(\Delta_0+2\Gamma)$.
The equality $h^*\Delta_0=2E$ implies that $h^*\Gamma=\sigma^*D-E$.
Thus we obtain
$$K_{\widehat{X}}=h^*(-2\Delta_0-4\Gamma)+R=-4E-4(\sigma^*D-E)+R=-2\sigma^*L+R.$$
As in Case B-I, we have
$$B\Delta_0=Rh^*\Delta_0=(K_{\widehat{X}}+2\sigma^*L)(2E)
=2(\sigma^*K_X+E+2\sigma^*L)E=-2,$$
$$B\Gamma=Rh^*\Gamma=(\sigma^*K_X+E+2\sigma^*L)(\sigma^*D-E)=K_XD+2LD-E^2=8.$$
Hence $B\sim 8\Delta_0+14\Gamma$. This equality implies
$$K_{\widehat{X}}=h^*(4\Delta_0+7\Gamma)-Z-2\sigma^*L=3\sigma^*D+E-Z,$$
i.e., $K_X\sim 3D-\sigma_*Z$. Since $K_X\equiv 3D$, we get $K_X\sim
3D$ and $Z=0$. Let $\overline{X}$ be the canonical resolution of the
double covering $h$ and $m_i$ the multiplicity of the corresponding
singularity. By the standard theory of double covering, we obtain
\begin{eqnarray*}
\chi(\mathcal{O}_{\overline{X}})=\chi(\mathcal{O}_{\widehat{X}})& = & \frac{1}{4}B\left(-2\Delta_0-4\Gamma+\frac{1}{2}B\right)+2-\sum_i\frac{1}{2}\left[\frac{m_i}{2}\right]\left(\left[\frac{m_i}{2}\right]-1\right)\\
                                                                & = & 7-\frac{1}{2}\sum_i\left[\frac{m_i}{2}\right]\left(\left[\frac{m_i}{2}\right]-1\right),
\end{eqnarray*}
\begin{eqnarray*}
K_{\overline{X}}^2=2\left(K_{\Sigma_2}+\frac{1}{2}B\right)^2-\sum_i2\left(\left[\frac{m_i}{2}\right]-1\right)^2=8-2\sum_i\left(\left[\frac{m_i}{2}\right]-1\right)^2.
\end{eqnarray*}
The equality $q(\widehat{X})=q(\overline{X})=q(X)=0$ implies that
$$p_g(X)=6-\frac{1}{2}\sum_i\left[\frac{m_i}{2}\right]\left(\left[\frac{m_i}{2}\right]-1\right).$$

Since $p_g(X)=6$, we have $[m_i/2]([m_i/2]-1)=0$ for all indices.
Thus $B\in|8\Delta_0+14\Gamma|$ is a reduced curve with at worst
negligible singularities. Therefore, in this case, $X$ is the
surface of type 3 in the theorem.

Case C. $K_X^2=L^2+6$. From (\ref{4}), it follows that
$L^2+6=K_X^2\leq L^2+4/L^2+4$. This implies $L^2\leq 2$. Thus
$L^2=2$ and $K_X^2=8$. By (\ref{*}), we get $h^0(L)=3$ and $K_XL=4$.
Hence $K_X^2L^2=16=(K_XL)^2$. By Hodge's index theorem, we conclude
that $K_X\equiv 2L$.

Since $h^0(L)=3$, we know that $\varphi_L:X\rightarrow
W=\mathbf{P}^2$ is generically $2$ to $1$. Let $m_i$ be the
multiplicity of the corresponding singularity. $R$ and $B$ denote,
respectively, the ramification divisor and the branch locus of
$\varphi_L$. If $H$ denotes a line on $\mathbf{P}^2$, then we have
$K_X=\varphi_L^*(-3H)+R=-3L+R$. By the theory of double covering,
there exists an effective divisor $Z$ on $X$ such that
$2R=\varphi_L^*B-2Z$ and $LZ=0$. We get $BH=RL=(K_X+3L)L=10$, i.e.,
$B\sim 10H$. Thus $K_X\sim -3L+5L-Z\sim 2L-Z$. Since $K_X\equiv 2L$,
we obtain $K_X\sim 2L$ and $Z=0$. By \cite[Lemma 5]{Hor}, we know
that $B\in|10H|$ is a reduced curve with at worst negligible
singularities. Therefore $X$ is the surface of type 4 in the theorem
and $p_g(X)=6-(1/2)\sum[m_i/2]([m_i/2]-1)=6$.
\end{proof}

Now, we assume $L$ computes $\beta(X)$ and $|L|$ has no fixed part.

\begin{theorem}\label{theorem3.2}
If $\beta(X)=0$, then $|K_X|$ is composed of a rational pencil and
$|L|$ is a sum of some fibers of the pencil.
\end{theorem}
\begin{proof}
It is just a special case of Theorem \ref{theorem0.2}.
\end{proof}
When $\alpha$ or $\beta$ increase, the surface become more and more
complicated and we can not hope to give a detailed description of
the surface.

\section{The number of moduli of a surface}\label{sec:5}
\begin{definition}
For a surface of general type $S$, we define $M(S)$, which is the
number of moduli of $S$, to be the dimension of its Kuranishi space
$B$, i.e., the maximum of the dimensions of the irreducible
components of $B$ (cf. \cite{Cat1}).
\end{definition}

Hence we have
$$10\chi(\mathcal{O}_S)-2K_S^2=h^1(T_S)-h^2(T_S)\leq
M(S)=\dim B\leq h^1(T_S).$$ By \cite{Bo}, we have
$h^0(T_S)=h^0(\Omega_S^1(-K_S))=0$. By Serre duality,
$h^2(T_S)=h^0(\Omega_S^1(K_S))$, and we have
\begin{equation}\label{8}
M(S)\leq h^1(T_S)=10\chi(\mathcal{O}_S)-2K_S^2+h^0(\Omega_S^1(K_S)).
\end{equation}
Hence one can give an upper bound for $M(S)$ by giving an upper
bound for $h^0(\Omega_S^1(K_S))$. The following theorem improves the
inequality given in \cite[Theorem B]{Cat2}.

\begin{theorem}\label{theorem4.2}Let $X$ be a smooth minimal complex projective surface
of general type. We have the inequality
$M(X)\leq10\chi(\mathcal{O}_X)+(5/2)K_X^2+4$. Furthermore, if
$q(X)>0$, then $M(X)\leq10\chi(\mathcal{O}_X)+(1/2)K_X^2+4$.
\end{theorem}
\begin{proof}
We can assume $h^0(\Omega_X^1(K_X))>0$. We know that
$\Omega_X^1(K_X)$ is $K_X$-semistable (cf. \cite[Corollary 1.2]{E},
\cite{Bo} or \cite{T}). Thus we can find an invertible sheaf
$\mathcal{O}_X(L)$ such that it is an maximal invertible subbundle
of $\Omega_X^1(K_X)$ of maximal slope. Then
$\Omega_X^1(K_X)/\mathcal{O}_X(L)$ is torsion free and
$(\Omega_X^1(K_X)/\mathcal{O}_X(L))^{\vee\vee}\cong\mathcal{O}_X(3K_X-L)$.
Hence we obtain
\begin{equation}\label{9}
h^0(\Omega_X^1(K_X))\leq
h^0(L)+h^0((\Omega_X^1(K_X)/\mathcal{O}_X(L)))\leq
h^0(L)+h^0(3K_X-L).
\end{equation}

Case 1. $K_XL<K_X^2$. The inequality implies $K_X(3K_X-L)>2K_X^2$.
By the assumption $h^0(\Omega_X^1(K_X))>0$, we have $K_XL\geq 0$.
Thus by Proposition \ref{proposition1.5}, we get $h^0(L)\leq
K_XL/2+2$ and $h^0(3K_X-L)\leq (K_X(3K_X-L))^2/2K_X^2+2$. It follow
from (\ref{9}) that
\begin{eqnarray*}
h^0(\Omega_X^1(K_X)) & \leq & \frac{K_XL}{2}+2+\frac{(K_X(3K_X-L))^2}{2K_X^2}+2\\
                     &   =  & \frac{(K_XL-(5/2)K_X^2)^2-(25/4)(K_X^2)^2}{2K_X^2}+\frac{9}{2}K_X^2+4\\
                     & \leq & \frac{(0-(5/2)K_X^2)^2-(25/4)(K_X^2)^2}{2K_X^2}+\frac{9}{2}K_X^2+4\\
                     &   =  & \frac{9}{2}K_X^2+4.
\end{eqnarray*}
Hence $M(X)\leq
10\chi(\mathcal{O}_X)-2K_X^2+h^0(\Omega_X^1(K_X))\leq
10\chi(\mathcal{O}_X)+(5/2)K_X^2+4$.

Case 2. $K_XL\geq K_X^2$. By Proposition \ref{proposition1.5}, we
get $h^0(L)\leq(K_XL)^2/2K_X^2+2$. Since $\Omega_X^1(K_X)$ is
$K_X$-semistable, we have $K_XL\leq(3/2)K_X^2$. Hence
$K_X(3K_X-L)\geq(3/2)K_X^2$. By Proposition \ref{proposition1.5}, we
obtain $h^0(3K_X-L)\leq(K_X(3K_X-L))^2/2K_X^2+2$. From (\ref{9}), it
follows that
\begin{eqnarray*}
h^0(\Omega_X^1(K_X)) & \leq & \frac{(K_XL)^2}{2K_X^2}+2+\frac{(K_X(3K_X-L))^2}{2K_X^2}+2\\
                     &   =  & \frac{(K_XL-(3/2)K_X^2)^2-(9/4)(K_X^2)^2}{K_X^2}+\frac{9}{2}K_X^2+4\\
                     & \leq & \frac{(K_X^2-(3/2)K_X^2)^2-(9/4)(K_X^2)^2}{K_X^2}+\frac{9}{2}K_X^2+4\\
                     &   =  & \frac{5}{2}K_X^2+4.
\end{eqnarray*}
Hence $M(X)\leq
10\chi(\mathcal{O}_X)-2K_X^2+h^0(\Omega_X^1(K_X))\leq
10\chi(\mathcal{O}_X)+(1/2)K_X^2+4$.

When $q(X)=h^0(\Omega_X^1)>0$, we know that
$\mathcal{O}_X(K_X)\subset\Omega_X^1(K_X)$. Thus $K_XL\geq K_X^2$.
Therefore, by Case 2, we have $M(X)\leq
10\chi(\mathcal{O}_X)+(1/2)K_X^2+4$.
\end{proof}

\bigskip
\address{ 
Department of Mathematics and Statistics\\
Huazhong Normal University \\
Wuhan 430079 \\
People's Republic of China}{hsun@mail.ccnu.edu.cn}

\address{ 
Department of Mathematics\\
East China Normal University \\
Shanghai 200241 \\
People's Republic of China}{suntju@sohu.com}
\end{document}